\documentclass[11pt, reqno]{amsart}
\usepackage{amsmath, amsthm, amscd, amsfonts, amssymb, graphicx, color}
\usepackage[bookmarksnumbered,plainpages]{hyperref}
\newtheorem{theorem}{Theorem}
\newtheorem{lemma}[theorem]{Lemma}
\newtheorem{proposition}[theorem]{Proposition}
\newtheorem{corollary}[theorem]{Corollary}
\theoremstyle{definition}

\theoremstyle{remark}

\newcommand{\2}{{\mathcal A}^{**}}

\newcommand{\A}{\mathcal A}
\newcommand{\X}{\mathcal X}
\newcommand{\Y}{\mathcal Y}
\newcommand{\Z}{\mathcal Z}

\newcommand{\bea}{\begin{eqnarray*}}
\newcommand{\eea}{\end{eqnarray*}}
\begin{document}

\title []{Lifting derivations and  $n-$weak amenability of the second dual of a Banach algebra}
\author{S. Barootkoob}
\address{Department of Pure Mathematics, Ferdowsi University of Mashhad, P.O. Box 1159, Mashhad 91775, Iran} \email{sbk-923@yahoo.com}
\author{H.R. Ebrahimi Vishki}
\address{Department of Pure Mathematics and  Centre of Excellence
in Analysis on Algebraic Structures (CEAAS), Ferdowsi University of
Mashhad, P.O. Box 1159, Mashhad 91775, Iran.}
\email{vishki@um.ac.ir}

\begin{abstract} We show that for $n\geq 2$,  $n-$weak amenability of the second dual $\A^{**}$ of a Banach algebra $\A$ implies that of $\A$. We also provide a positive answer for the case $n=1,$ which sharpens some older results. Our method of proof also provides a unified approach to give short proofs for  some  known results in the case where $n=1$.
\end{abstract}
\maketitle
 The concept of $n$-weak amenability was initiated and intensively developed by Dales,
Ghahramani and Gronb\ae k \cite{DGG}. A Banach algebra ${\mathcal
A}$ is said to be $n$-weakly amenable ($n\in\mathbb{N}$) if every
(bounded) derivation from ${\mathcal A}$ into ${\mathcal A}^{(n)}$ (the $n^{\rm th}$ dual of $\A$) is
inner. Trivially, $1$-weak amenability is nothing else than weak
amenability, which  was first introduced and intensively studied by Bade, Curtis
and Dales \cite{BCD} for commutative Banach algebras and then by
Johnson  \cite{J} for a general Banach algebra.

We equip  the second dual $\2$
of $\A$ with its first Arens product and focus on the following question which is of special interest, especially for the case when $n=1$.
\begin{center}
{\it Does  $n-$weakly amenability of $\2$   force $\A$ to be $n-$weakly amenable?}
\end{center}
In the present paper first we shall show:
\begin{theorem}\label{1} The answer to the above question is positive for any $n\geq 2$.
\end{theorem}
Then we consider the case $n=1$, which is a long-standing open problem with  a slightly different feature from that of  $n\geq 2$. This case   has been investigated and partially answered by many authors (see Theorem \ref{3}, in which we rearrange some known answers from \cite{EF, GL, GLW, JMV}). As a consequence  of our general  method of proof (for the case $n=1$),  we  present the next positive answer; in which, $\pi$ denotes the product of $\A,$ $\pi^*:\A^*\times\A\rightarrow\A^*$ is defined by $$\langle\pi^*(a^*, a), b\rangle=\langle a^*,\pi(a,b)\rangle, \ \ \ (a^*\in\A^*, a,b\in\A),$$ and $Z_\ell(\pi^*)$ is the left topological centre of $\pi^*,$ (see the next section).
\begin{theorem}\label{2} Let $\A$ be a Banach algebra such that every derivation $D:\A\rightarrow \A^*$ satisfies  $D^{**}(\A^{**})\subseteq Z_\ell(\pi^*)$. Then  weak amenability of $\2$ implies that of $\A$.
\end{theorem}
 As a rapid consequence we get the next result, part $(ii)$ of which sharpens  {\cite[Corollary 7.5]{DPV}} and also {\cite[Theorem 2.1]{EF}} (note that $WAP(\A)\subseteq \A^*\subseteq Z_\ell(\pi^*)$); indeed, it  shows that the hypothesis of Arens regularity of $\A$ in \cite[Corollary 7.5]{DPV} is superfluous.
\begin{corollary}\label{3'}  For a Banach algebra $\A$, in either of the following cases,  the weak amenability of $\A^{**}$ implies that of $\A.$

$(i)$ If $\pi^*$ is Arens regular.

$(ii)$ If every derivation from $\A$ into $\A^*$ is weakly compact.
\end{corollary}
The influence of the impressive paper {\cite{GLW}} of Ghahramani {\it et al.}  on our work should be evident. It should finally be remarked that part $(ii)$ of  Corollary \ref{3'} actually demonstrates what Ghahramani {\it et al.} claimed in a remark  following {\cite[Theorem 2.3]{GLW}}. Indeed, as we shall see in the proof of Theorem \ref{2},   ${J_0}^*\circ D^{**}$ is a derivation ($J_0:\A\rightarrow\A^{**}$ denotes  the canonical embedding),  however they claimed that $D^{**}$ is a derivation and in their calculation of limits  they used the Arens regularity of $\A$; see also a remark just after the proof of {\cite[Corollary 7.5]{DPV}}.\\
\section*{The proofs}
To prepare the proofs, let us first fix some notations and preliminaries.
 Following the seminal work ~\cite{A} of  Arens,   every bounded bilinear map $f: {\X}\times {\Y}\rightarrow {\Z}$  (on normed spaces) has two natural but, in general, different  extensions  $f^{***}$ and $f^{r***r}$  from $\X^{**}\times\Y^{**}$  to $\Z^{**}$. Here  the flip map $f^{r}$ of $f$ is defined  by $f^{r}(y,x)= f(x,y),$
the adjoint  $ f^*:{\Z}^*\times {\X} \rightarrow {\Y}^*$ of $f$  is defined by
 $$\langle f^*(z^*, x),y\rangle=\langle z^*,f(x,y)\rangle \ \ \ \ \ \   (x\in {\X},y\in {\Y}\ \ {\rm and}\ \  z^*\in {\Z}^*);$$ and also the second and  third adjoints $f^{**}$ and $f^{***}$ of $f$ are defined by   $f^{**}=(f^*)^*$ and $f^{***}=(f^{**})^*$, respectively. Continuing this process one can define  the  higher  adjoints $f^{(n)},$  $(n\in {\mathbb N})$.

  We also define the left topological centre $Z_\ell(f)$ of $f$ by \bea Z_\ell(f)&=&\{x^{**}\in {\X}^{**}; y^{**}\longrightarrow f^{***}(x^{**},y^{**}) : {\Y}^{**}\longrightarrow {\Z}^{**} \ {\rm is} \ w^{*}-w^{*}-{\rm continuous} \}.\eea
A bounded bilinear mapping $f$ is said to be Arens regular if $f^{***}=f^{r***r}$, or equivalently  $Z_\ell(f)={\X}^{**}.$

It should be remarked that, in the case where $\pi$ is the multiplication of a Banach algebra $\A,$ then $\pi^{***}$ and $\pi^{r***r}$ are
actually the
 first and second Arens products on $\2$,
respectively.
From now on, we only deal with the first Arens product $\Box$ and our results are based on $(A^{**}, \Box)$. Similar results can be derived if one uses the second Arens product instead of the first one.

Consider $\A$ as a Banach $\A$-module equipped with its own multiplication $\pi$. Then $({\pi}^{r*r},\A^*,\pi^*)$ is the natural dual   Banach $\A-$module, in which, ${\pi}^{r*r}$ and $\pi^*$ denote its left and right module actions, respectively. Similarly, the $n^{\rm th}$ dual  $\A^{(n)}$ of $\A$ can be made into a Banach $\A-$module in a natural fashion.  A direct verification reveals that $({\pi}^{(3n)}, \A^{(2n)}, \pi^{(3n)})$  is a Banach $\A^{**}-$module.  It  induces the natural dual Banach $\A^{**}-$module $({\pi}^{(3n)r*r}, \A^{(2n+1)}, \pi^{(3n+1)})$ which will be used in the sequel.  Note that we have also $({\pi}^{r*r(3n)}, \A^{(2n+1)}, \pi^{(3n+1)})$  as a Banach $\A^{**}-$module that induced by $({\pi}^{r*r},\A^*,\pi^*)$. It should be mentioned  that these two  actions on $\A^{(2n+1)}$ do not coincide, in general. For more information on the equality of these actions in the case where $n=1$ see \cite{DPV, MV}.

 From now on, we identify (an element of) a normed space  with its canonical image in its second dual; however,   we also use $J_n:\A^{(n)}\rightarrow \A^{(n+2)}$ for  the canonical embedding.

 We commence with the next lemma.
\begin{lemma}\label{4} Let $\A$ be a Banach algebra, $n\in\mathbb{N}$ and let $D:\A\rightarrow \A^{(2n-1)}$ be a derivation.

 (i) If $n\geq 2$ then $[{(J_{2n-2})}^*\circ D^{**}]: \2\rightarrow {\A}^{(2n+1)}$ is a derivation.

 (ii) For $n=1$, $[{J_0}^*\circ D^{**}]: \2\rightarrow{\A}^{***}$ is a derivation if and only if $\pi^{***r*}(D^{**}(\A^{**}),\A)\subseteq \A^*$.
\end{lemma}
\begin{proof} $(i).$ It is enough to show that for any $a^{**}, b^{**}\in A^{**}$
$$[(J_{2n-2})^*\circ D^{**}](a^{**}\Box \ b^{**})=\pi^{(3n+1)}([(J_{2n-2})^*\circ D^{**}](a^{**}),b^{**})+ \pi^{(3n)r*r}(a^{**},[(J_{2n-2})^*\circ D^{**}](b^{**})).$$
To this end let $\{a_{\alpha}\}$ and  $\{b_{\beta}\}$ be bounded nets in $\A,$
$w^{*}-$converging to $a^{**}$ and $b^{**}$, respectively. Then
\bea
D^{**}(a^{**}\Box b^{**})&=&w^*-\lim_\alpha w^*-\lim_\beta D(a_\alpha b_\beta)\\
&=&w^*-\lim_\alpha w^*-\lim_\beta [\pi^{(3n-2)}(D(a_\alpha),b_\beta)+\pi^{(3n-3)r*r}(a_\alpha,D(b_\beta))]\\
&=&\pi^{(3n+1)}(D^{**}(a^{**}),b^{**})+\pi^{(3n-3)r*r***}(a^{**},D^{**}(b^{**})).
\eea
For each $a^{(2n-2)}\in \A^{(2n-2)},$

$\langle (J_{2n-2})^*(\pi^{(3n-3)r*r***}(a^{**}, D^{**}(b^{**}))),a^{(2n-2)}\rangle=$
\bea  \ \ \ \ \ \ \  \ \ \ \ \ \ \ \ \ \ \ \ \ \  \ \ \ \ \ \ \ \ \  \ \ \ \ \ \ \ &=&\lim_\alpha\lim_\beta \langle D(b_\beta),\pi^{(3n-3)}(a^{(2n-2)},a_\alpha)\rangle\\
&=&\lim_\alpha\langle D^{**}(b^{**}),\pi^{(3n-3)}(a^{(2n-2)},a_\alpha)\rangle\\
&=&\lim_\alpha\langle [(J_{2n-2})^*\circ D^{**}](b^{**}),\pi^{(3n-3)}(a^{(2n-2)},a_\alpha)\rangle\\
&=&\langle [(J_{2n-2})^*\circ D^{**}](b^{**}),\pi^{(3n)}(a^{(2n-2)},a^{**})\rangle\\
&=&\langle\pi^{(3n)r*r}(a^{**},[(J_{2n-2})^*\circ D^{**}](b^{**})),a^{(2n-2)}\rangle.
\eea

 Since for $n\geq 2$, $$\pi^{(3n)}(\A^{**},\A^{(2n-2)})=\pi^{(3n-3)}(\A^{**},\A^{(2n-2)})\subseteq\pi^{(3n-3)}(\A^{(2n-2)},\A^{(2n-2)})\subseteq\A^{(2n-2)}$$
(note that the same inclusion may not valid  for the case $n=1$; indeed, it holds if and only if $\pi^{***}(\A^{**},\A)\subseteq\A$, or equivalently, $\A$ is a left ideal in $\2$!), we get $\pi^{(3n)}(b^{**},a^{(2n-2)})\in\A^{(2n-2)}$ and so

$\langle (J_{2n-2})^*(\pi^{(3n+1)}(D^{**}(a^{**}),b^{**})),a^{(2n-2)}\rangle=$
\bea
\ \ \ \ \ \ \ \ \ \ \ \ \ \ \ \ \ \ \ \ \ \ \ \ \ \ \ \ \ \ \ \  &=&\langle D^{**}(a^{**}),\pi^{(3n)}(b^{**},a^{(2n-2)})\rangle\\
&=&\langle [(J_{2n-2})^*\circ D^{**}](a^{**}),\pi^{(3n)}(b^{**},a^{(2n-2)})\rangle\\
&=&\langle\pi^{(3n+1)}([(J_{2n-2})^*\circ D^{**}](a^{**}),b^{**}),a^{(2n-2)}\rangle.
\eea
Therefore

$[(J_{2n-2})^*\circ D^{**}](a^{**}\Box b^{**})=$
\bea  \ \ \ \ \ \ \ \ \  \  &=&(J_{2n-2})^*(\pi^{(3n+1)}(D^{**}(a^{**}),b^{**}))+ (J_{2n-2})^*(\pi^{(3n-3)r*r***}(a^{**},D^{**}(b^{**})))\\
&=&\pi^{(3n+1)}([(J_{2n-2})^*\circ D^{**}](a^{**}),b^{**})+ \pi^{(3n)r*r}(a^{**},[(J_{2n-2})^*\circ D^{**}](b^{**}));
\eea  as required.\\

For $(ii)$, examining the above proof for the case $n=1$  shows that,  ${J_0}^*\circ D^{**}:\2\rightarrow \A^{***}$ is a derivation if and only if
 $${J_0}^*(\pi^{****}(D^{**}(a^{**}),b^{**}))=\pi^{****}([{J_0}^*\circ D^{**}](a^{**}),b^{**})\ \ \ (a^{**}, b^{**}\in \A^{**}),$$
 which holds if and only if
  $$\langle\pi^{****}(D^{**}(a^{**}),b^{**}),a\rangle=\langle\pi^{****}([{J_0}^*\circ D^{**}](a^{**}),b^{**}),a\rangle\ \ \ \ \ \ (a\in \A);$$ or equivalently, $$\langle\pi^{***r*}(D^{**}(a^{**}),a),b^{**}\rangle=\langle\pi^{***r*}([{J_0}^* \circ D^{**}](a^{**}),a),b^{**}\rangle.$$
 As $\pi^{***r*}([{J_0}^* \circ D^{**}](a^{**}),a)=\pi^{**r}([{J_0}^* \circ D^{**}](a^{**}),a)\in\A^*$ and also $\pi^{***r*}(D^{**}(a^{**}),a)_{|\A}=\pi^{**r}([{J_0}^* \circ D^{**}](a^{**}),a)$; the map  $[{J_0}^*\circ D^{**}]: \2\rightarrow{\A}^{***}$ is a derivation if and only if   $\pi^{***r*}(D^{**}(a^{**}),a)\in\A^*$, as claimed.\\
\end{proof}
We are now ready to present the proofs of the main results.\\

 \noindent{\bf Proof of Theorem \ref{1}.}
Let $n\in\mathbb{N}$,   $D:\A\rightarrow \A^{(2n)}$ be a derivation and let $a^{**}, b^{**}\in A^{**}.$  As $(\pi^{(3n+3)},\A^{(2n+2)},\pi^{(3n+3)})$ is a Banach $\2-$module, a  standard double limit process argument$-$similar to what has been used at  the beginning of the proof of the preceding lemma$-$ shows  that   $D^{**}:\A^{**}\rightarrow \A^{(2n+2)}$ satisfies
$$D^{**}(a^{**}\Box b^{**})=\pi^{(3n+3)}(D^{**}(a^{**}),b^{**})+\pi^{(3n+3)}(a^{**},D^{**}(b^{**})).$$
Therefore $D^{**}$ is a derivation and so   $(2n)-$weak amenability of $\2$ implies that  $D^{**}=\delta_{a^{(2n+2)}}$ for some $a^{(2n+2)}\in \A^{(2n+2)}$. Now we get  $D=\delta_{{(J_{2n-1})^*}(a^{(2n+2)})}.$ Thus $D$ is inner and so $\A$ is $(2n)-$weakly amenable.

Now for the odd case, suppose that $\A^{**}$ is $(2n-1)-$weakly amenable and let    $D:\A\rightarrow \A^{(2n-1)}$ be a derivation. Then as we have seen in Lemma \ref{4}, when $n\geq2$ the mapping  $[{(J_{2n-2})}^*\circ D^{**}]: \2\rightarrow {\A}^{(2n+1)}$ is a derivation. But then, by the assumption,  $[{(J_{2n-2})}^*\circ D^{**}]=\delta_{a^{(2n+1)}}$ for some $a^{(2n+1)}\in \A^{(2n+1)}.$ It follows that  $D=\delta_{{(J_{2n-2})^*}(a^{(2n+1)})}$, so that $D$ is inner, as claimed.  \hspace{6cm}\ \ \ \ \ \ \ \ \ \ \ \ \   \ $\Box$\\

\noindent{\bf Proof of Theorem \ref{2}.} Let $a^{**},b^{**}\in\A^{**}, a\in \A$ and let  $\{a^{**}_{\alpha}\}$  be a  net in $\A^{**}$
$w^{*}-$converging to $a^{**}.$ As $D^{**}(b^{**})\in Z_\ell(\pi^*),$
\bea
\lim_\alpha<\pi^{***r*}(D^{**}(b^{**}),a),a^{**}_\alpha>&=&\lim_\alpha<D^{**}(b^{**}),\pi^{***}(a^{**}_\alpha,a)>\\
&=&\lim_\alpha<\pi^{****}(D^{**}(b^{**}),a^{**}_\alpha),a>\\
&=&<\pi^{****}(D^{**}(b^{**}),a^{**}),a>\\
&=&<D^{**}(b^{**}),\pi^{***}(a^{**},a)>\\
&=&<\pi^{***r*}(D^{**}(b^{**}),a),a^{**}>.
\eea
And this means that  $\pi^{***r*}(D^{**}(b^{**}),a)\in\A^*$, so that  ${J_0}^*\circ D^{**}$ is derivation by Lemma \ref{4}. Now by the assumption ${J_0}^*\circ D^{**}=\delta_{a^{***}},$ for some $a^{***}\in\A^{***}$, and this follows that $D=\delta_{{J_0}^*(a^{***})},$ so that $\A$ is weakly amenable.     \hspace{10cm}  \ \ \ \ \ \ \ \ \ \  \ $\Box$\\
\section*{Further consequences}
Recall  that for a derivation $D:\A\rightarrow\A^*$ the second adjoint $D^{**}$ is a derivation if and only if  $\pi^{r*r***}(a^{**},D^{**}(b^{**}))=\pi^{***r*r}(a^{**},D^{**}(b^{**})),$ for every $a^{**},b^{**}\in\2;$
or equivalently, $\pi^{****}(D^{**}(\A^{**}),\A^{**})\subseteq \A^*$; see {\cite[Theorem 7.1]{DPV}} and also {\cite[Theorem 4.2]{MV}} for a more general case. While, as Lemma  \ref{4}   demonstrates,  ${J_0}^*\circ D^{**}$ is a derivation if and only if $\pi^{***r*}(D^{**}(\A^{**}),\A)\subseteq \A^*$. In the next result we investigate the interrelation between   $D^{**}$ and ${J_0}^*\circ D^{**}.$
\begin{proposition}\label{2'}  Let $D:\A\rightarrow\A^*$ be a derivation.

$(i)$ If $D^{**}$ is a derivation and $\A^{**}\Box \A=\A^{**}$ then ${J_0}^*\circ D^{**}$ is a derivation.

$(ii)$ If ${J_0}^*\circ D^{**}$ is a derivation and $\A$ is Arens regular then $D^{**}$ is a derivation.
\end{proposition}
\begin{proof}$(i)$. As $\A^{**}\Box \A=\A^{**},$ for each $b^{**}\in\2$ there exist $a^{**}\in\2$ and  $a\in\A$
such that $a^{**}\Box a=b^{**}.$  Then
\bea
\pi^{***r*}(D^{**}(b^{**}),b)&=&\pi^{***r*}(D^{**}(a^{**}\Box a),b)\\
&=&\pi^{***r*}(\pi^{****}(D^{**}(a^{**}),a)+\pi^{***r*r}(a^{**},D(a)),b)\\
&=&\pi^{r*}(\pi^{****}(D^{**}(a^{**}),a)+\pi^{**}(a^{**},D(a)),b)\in \A^*.
\eea
It follows from Lemma \ref{4} that ${J_0}^*\circ D^{**}$ is a derivation.

$(ii)$. Since  ${J_0}^*\circ D^{**}$ is a derivation, $${J_0}^*(\pi^{****}(D^{**}(a^{**}),b^{**}))=\pi^{****}([{J_0}^*\circ D^{**}](a^{**}),b^{**})\ \ \ \ (a^{**},b^{**}\in\A^{**}).$$
Let  $\{a_{\alpha}\}$  be a bounded net in $\A,$
$w^{*}-$converging to $a^{**}.$ Then as $\A$ is Arens regular,
\bea
\langle\pi^{r*r***}(a^{**},D^{**}(b^{**})),c^{**}\rangle&=&\lim_\alpha\langle\pi^{r*r**}(D^{**}(b^{**}),c^{**}),a_\alpha\rangle\\
&=&\lim_\alpha\langle{J_0}^*(\pi^{****}(D^{**}(b^{**}),c^{**})),a_\alpha\rangle\\
&=&\lim_\alpha\langle\pi^{****}([{J_0}^*\circ D^{**}](b^{**}),c^{**}),a_\alpha\rangle\\
&=&\lim_\alpha\langle[{J_0}^*\circ D^{**}](b^{**}),\pi^{***}(c^{**},a_\alpha)\rangle\\
&=&\langle[{J_0}^*\circ D^{**}](b^{**}),\pi^{***}(c^{**},a^{**})\rangle\\
&=&\langle\pi^{***r*r}(a^{**},D^{**}(b^{**})),c^{**}\rangle,
\eea
for all $c^{**}\in\A^{**}.$ Therefore $D^{**}$ is a derivation.
\end{proof}
As a by-product of our method of proof we provide a unified approach to new proofs for  some known results for the case where $n=1.$
\begin{theorem}\label{3} In either of the following cases, weak amenability of $\2$ implies that of $\A$.

$(i)$  $\A$ is a left ideal in $\2$; {\cite[Theorem 2.3]{GLW}}.

$(ii)$  $\A$ is a dual Banach algebra; {\cite[Theorem 2.2]{GL}}.

$(iii)$ $\A$ is a right ideal in $\2$ and $\2\Box \A=\2$; {\cite[Theorem 2.4]{EF}}.
\end{theorem}
\begin{proof} In either cases, it sufficients  to show that for a derivation $D:\A\rightarrow\A^*$ the map ${J_0}^*\circ D^{**}: \2\rightarrow{\A}^{***}$ is also a derivation, or equivalently,  $\pi^{***r*}(D^{**}(\A^{**}),\A)\subseteq \A^*.$

$(i)$ If $\A$ is a left ideal in $\2$, i.e. $\2\Box \A\subseteq\A$,  then  for each $a^{**}, b^{**}\in A^{**}, a\in\A,$ \bea\langle \pi^{***r*}(D^{**}(a^{**}),a), b^{**}\rangle=\langle D^{**}(a^{**}), b^{**}\Box a\rangle&=&\langle\pi^{***r*}([{J_0}^* \circ D^{**}](a^{**}),a), b^{**}\rangle\\&=&\langle\pi^{**r}([{J_0}^* \circ D^{**}](a^{**}),a),b^{**}\rangle.\eea
Therefore $\pi^{***r*}(D^{**}(a^{**}),a)=\pi^{**r}([{J_0}^* \circ D^{**}](a^{**}),a)\in \A^*,$ as required.

$(ii)$ Let $\A$ be a dual Banach algebra with a predual $\A_*$. It is easy to verify  that ${J_0}^*\circ D^{**}=D\circ(J_{\A_*})^*$, where $J_{\A_*}:\A_*\rightarrow\A^{*}$ denotes the canonical embedding. Now using the fact that $(J_{\A_*})^*:\A^{**}\rightarrow\A$ is a homomorphism, a direct verification shows that $D\circ(J_{\A_*})^*$ is a derivation.

$(iii)$ To show that ${J_0}^*\circ D^{**}: \2\rightarrow{\A}^{***}$ is a derivation, by Proposition \ref{2'} we only need to show that $D^{**}$ is a derivation. However  this was done in the proof of   {\cite[Theorem 2.4]{EF}}, we also give the  next somewhat shorter proof for this. Let  $a^{**}, b^{**}, c^{**}, d^{**}\in\2$ and $a\in\A$  such that $d^{**}\Box a=b^{**}.$ As $a\Box c^{**}\in\A,$                                                                                \bea
\langle\pi^{****}(D^{**}(a^{**}),b^{**}), c^{**}\rangle&=&\langle\pi^{****}(D^{**}(a^{**}), d^{**}\Box a), c^{**}\rangle\\&=&\langle\pi^{****}(D^{**}(a^{**}),d^{**}), a\Box c^{**}\rangle\\
&=&\langle\pi^*({J_0}^*(\pi^{****}(D^{**}(a^{**}),d^{**})), a), c^{**}\rangle.
\eea
 We have thus $\pi^{****}(D^{**}(a^{**}),b^{**})=\pi^*({J_0}^*(\pi^{****}(D^{**}(a^{**}),d^{**})), a)\in\A^{*}$, and this says that $\pi^{****}(D^{**}(\A^{**}),\A^{**})\subseteq \A^*$, as required.
\end{proof}
\section*{Acknowledgments}
The useful comments of anonymous referee is gratefully acknowledged.

\end{document}